\documentclass[10pt,reqno]{amsart}

\usepackage{xcolor}
\usepackage{makecell}

\usepackage{todonotes}

\usepackage{bbm}
\usepackage{tabu}
\usepackage{amsmath}
\usepackage{mathrsfs}
\usepackage{bm}
\usepackage{enumerate}
\usepackage{amsfonts} 
\usepackage{graphicx,latexsym,bm,amsmath,amssymb,verbatim,multicol,lscape}
\usepackage{enumitem}
\newtheorem{thm}{Theorem} [section]

\newtheorem{cor}[thm]{Corollary}
\newtheorem{lem}[thm]{Lemma}

\theoremstyle{definition}

\theoremstyle{remark}
\newtheorem{rem}[thm]{Remark}

\numberwithin{equation}{section}

\newcommand{\A}{\mathcal{A}}

\newcommand{\fq}{{\mathbb F}_{q}}
\newcommand{\fqr}{{\mathbb F}_{q^r}}

\newcommand{\ftwo}{{\mathbb F}_{2}}

\newcommand{\G}{{\mathcal G}}
\newcommand{\M}{{\mathcal M}}
\newcommand{\B}{{\mathcal B}}

\newcommand{\rmv}[1]{}

\def\<{\left\langle}
\def\>{\right\rangle}

\begin{document}

\title[
Multiplicative character sums over two classes of subsets]
{
Multiplicative character sums over two classes of subsets of quadratic extensions of finite fields}%
\author{Kaimin Cheng}
\address{School of Mathematics and Information, China West Normal University, Nanchong, 637002, P. R. China}
\email{ckm20@126.com}
\author{Arne Winterhof}
\address{Johann Radon Institute for Computational and Applied Mathematics, Austrian Academy of Science, Linz 4040, Austria}
\email{arne.winterhof@ricam.oeaw.ac.at}
\subjclass{Primary 11T23}%
\keywords{Finite fields, character sums, sparse elements, primitive elements.}
\date{\today}
\begin{abstract}
Let $q$ be a prime power and $r$ a positive even integer. Let $\fq$ be the finite field with $q$ elements and $\fqr$ be its extension field of degree $r$. Let $\chi$ be a nontrivial multiplicative character of $\fqr$ and $f(X)$ a polynomial over $\fqr$ with a simple root in~$\fqr$. In this paper, we improve estimates for character sums
$\sum\limits_{g \in\G}\chi(f(g))$,
where $\G$ is 
either a subset of $\fqr$ of sparse elements, with respect to some fixed basis of $\fqr$ which contains a basis of $\mathbb{F}_{q^{r/2}}$, or a subset avoiding affine hyperplanes in general position.
While such sums have been previously studied, our approach yields sharper bounds by reducing them to sums over the subfield $\mathbb{F}_{q^{r/2}}$ rather than sums over general linear spaces.
These estimates can be used to prove the existence of primitive elements in~$\G$ in the standard way. 
\end{abstract}

\maketitle
\section{Introduction}

Let $q$ be a power of a prime $p$ and $\fq$ be the finite field with $q$ elements. Let $r$ be a positive integer and $\fqr$ be the extension field over $\fq$ of degree $r$. Let $\chi$ be a nontrivial multiplicative character of $\fqr$, $\G$ be a subset of $\fqr$ and $f(X)$ be a  polynomial in $\fqr[X]$. Let  $S(\mathcal{G},\chi,f)$ be the character sum defined by 
$$S(\mathcal{G},\chi,f)=\sum_{\gamma\in\mathcal{G}}\chi(f(\gamma)),$$
and set $S(\mathcal{G},\chi)=S(\mathcal{G},\chi,f)$ if $f(X)=X$.
 For background on character sums over finite fields, see, for example, \cite{[LN],[MSW],[MP],[Wan],[W01]}.

Recently,  Grzywaczyk and the second author \cite{[GW]} provided an estimate for $S(\G',\chi)$, where 
\begin{itemize}
    \item
$\G'$ is a set of elements in $\fqr$ avoiding affine hyperplanes in general position (see the definition in Section 3),
\end{itemize}
which led to the solution of a problem of Fernandes and Reis \cite{[FR]} for $4\le q\le 5$ and sufficiently large $r$. This problem had already been settled in \cite{[FR]} for $q\ge 11$ and any~$r$ as well as for $7\le q\le 9$ and sufficiently large $r$. For $q=3$ and sufficiently large $r$ see \cite{[IS]} and \cite[Corollary 3.5]{[GW]}.

Moreover, M\'{e}rai, Shparlinski, and the second author \cite{[MSW]} obtained a bound for $S(\G'',\chi,f)$, where 
\begin{itemize}
\item $\G''$ is a set of sparse elements of $\fqr$ with a fixed weight (see the definition in Section 4).
\end{itemize}
Actually, \cite{[MSW]} deals with more general mixed character sums but here we deal with multiplicative character sums.

Now let $r$ be a positive even integer. We work with the extension field $\fqr$ 
and its subfield $\mathbb{F}_{q^{r/2}}$. Naturally,  $\mathbb{F}_{q^{r/2}}$ is also a subspace of  $\fqr$ over~$\fq$.
Throughout the paper, we assume that $\{\alpha_1,\ldots,\alpha_{r/2}\}$ is a basis of $\mathbb{F}_{q^{r/2}}$ over $\fq$, and $\{\alpha_1,\ldots,\alpha_r\}$ is a completion of $\{\alpha_1,\ldots,\alpha_{r/2}\}$ to a basis of $\fqr$ over $\fq$. In this paper, we focus on estimating the character sums $S(\G, \chi, f)$ for the two classes of subsets $\G=\G'$ and $\G=\G''$ of $\fqr$, where~$f(X)$ is a polynomial over $\fqr$ with a simple root in $\fqr$. Actually, by applying Wan’s bound \cite{[Wan]} on character sums over a subfield, more precisely the modified version of Martin and Yip \cite{[MY]}, we provide significant improvements over previous results on character sums obtained from bounds on character sums over general linear spaces. Note that we can also obtain bounds using any subfield $\mathbb{F}_{q^d}$ of $\mathbb{F}_{q^r}$ for some $d|r$. However, for~$d<\frac{r}{2}$ these bounds are not better than those obtained before via bounds for sums over subspaces, see for example \cite{[W01]}.

Our contributions can be summarized as follows.
\begin{itemize}[left=1.5em] 
\item $\G$ with restricted coordinates (Section 3):
We study the first class of subsets consisting of elements whose coordinates are restricted with respect to a given basis of $\fqr$ over $\fq$. For these subsets we derive a general estimate for the character sums. A direct application of this result leads to a substantial improvement of the work by Fernandes and Reis \cite{[FR]}, Iyer and Shparlinski \cite{[IS]}, and Grzywaczyk and the second author \cite{[GW]}
in our special case described above. In particular, for even and sufficiently large $r$ we strengthen their results by proving that a subset of the finite field $\fqr$ avoiding affine hyperplanes must contain a primitive element for much smaller even $r$, see Table \ref{table1} below. 
Note that for $q=2$ these sets consist of only one element.
\item $\G$ with $s$-sparse elements (Section 4): For $q=2$, Mérai, Shparlinski, and the second author \cite{[MSW]} investigated hybrid character sums over $\G$ with $s$-sparse elements, defined as 
$$S(\G,\chi,\psi,f_1,f_2)=\sum_{\gamma\in\G}\chi(f_1(\gamma))\psi(f_2(\gamma)),$$
where $\psi$ is an additive character of $\fqr$, and $f_1, f_2 \in \fqr(X)$. Under certain conditions, they established the bound
$$|S(\G,\chi,\psi,f_1,f_2)|\le 2^{\eta(\rho)r+o(r)},$$
where $\eta(\rho)$ is a function of $\rho = \frac{\min\{s, r-s\}}{r}$. For further details, we refer to \cite[Theorem 2.5]{[MSW]}. In this paper, for even integers $r$ and with $\psi = \psi_0$ (the trivial additive character), we improve the bound of 
\cite{[MSW]}
by providing a sharper estimate:
$$|S(\G,\chi,\psi_0,f_1,f_2)|\le 2^{\eta'(\rho)r+o(r)},$$
where $\eta'(\rho)$ is a function of $\rho$ defined in Theorem \ref{thm4.2}. For comparison, the trivial bound is given by
$$|S(\G,\chi,\psi,f_1,f_2)|\le 2^{H(\rho)r+o(r)},$$
where $H(x)$ denotes the binary entropy function defined by \eqref{entropy}. A comparison of these three estimates is illustrated in Figure \ref{fig1} below.
\begin{figure}[htbp]
    \centering
    \includegraphics[width=0.6\textwidth]{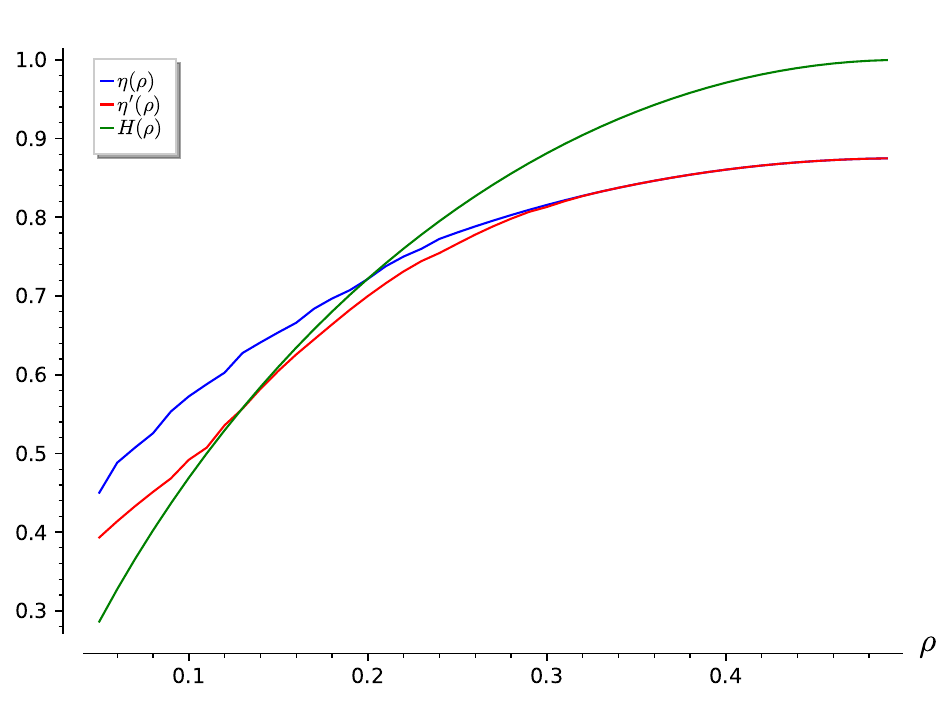}  
    \caption{Comparison of $\eta(\rho)$, $\eta'(\rho)$ and $H(\rho)$}
    \label{fig1}
\end{figure}
Our estimate clearly outperforms the previous bounds. Specifically, with the aid of computational tools, we obtain the following values:
$$\eta^{\prime}(0.13) \approx 0.55680, \quad \eta(0.13) \approx 0.62751, \quad H(0.13) \approx 0.55743,$$
and
$$\eta^{\prime}(0.32) \approx 0.82690, \quad \eta(0.32) \approx 0.82719, \quad H(0.32) \approx 0.90438.$$
This demonstrates the superiority of our estimate, particularly in the range $0.13 < \rho < 0.32$, where our estimate outperforms both the trivial bound and the previous estimate of \cite{[MSW]}.
However, our restrictions are stronger, in particular, we need $r$ even.
\end{itemize}

\section{Preliminaries}
In this section, we introduce some preliminary concepts and key results. 

We start with a corrected version of a result of Wan \cite[Corollary~2.4]{[Wan]} given in \cite[Corollary~3.5]{[MY]}.
Recall that the elements $\alpha,\alpha^q,\ldots,\alpha^{q^{r-1}}\in \fqr$ are {\em conjugates} over $\fq$.

\begin{lem}\label{lem2.1} 
Let $\chi_1,\dots,\chi_n$ be  
multiplicative characters of $\fqr$. Let $f_1(X),\ldots,f_n(X)$ be polynomials in $\fqr[X]$ such that no two of them share conjugated roots over $\fq$. Let~$D$ be the degree of the largest square-free divisor of $\prod\limits_{i=1}^nf_i(X)$. Suppose that for some $1\le i\le n$, there is a root $\xi_i$ with multiplicity $m_i$ of $f_i(X)$ such that the character $\chi_i^{m_i}$ is nontrivial on the set $Norm_{\fqr[\xi_i]/\fqr}(\fq[\xi_i])\setminus\{0\}$. Then, we have the estimate
$$\left|\sum_{a\in\fq}\chi_1(f_1(a))\cdots\chi_n(f_n(a))\right|\le (rD-1)q^{1/2}.$$
\end{lem}

Note that we dropped the condition that the polynomials $f_1(X),\ldots,f_n(X)$ are monic since the general case can be trivially reduced to this special case.

As usual we call $\alpha$ a {\em defining element} of $\fqr$ if $\fqr=\fq(\alpha)$ and denote by $\overline{\chi}$ the {\em conjugated character} of $\chi$.
We will use the following corollary. 

\begin{cor}\label{cor2.2}
Let $\chi$ be a nontrivial multiplicative character of $\fqr$. Let $f(X)$ be a polynomial over $\fqr$ having a simple root in $\fqr$.  Let $w_1$ and $w_2\in\fqr$ satisfy that $f(X+w_1)$ and $f(X+w_2)$ share no conjugated roots over $\fq$, and for a simple root $\xi$ of $f(X)$ in $\fqr$, either $\xi-w_1$ or $\xi-w_2$ is a defining element of $\fqr$ over $\fq$. Then we have
$$\left|\sum_{a\in\fq}\chi(f(a+w_1))\overline{\chi}(f(a+w_2))\right|\le(2rD-1)q^{1/2},$$
where $D$ is the degree of the largest square-free divisor of $f(X)$.
\end{cor}

\begin{proof}
Let $f_1(X)=f(X+w_1)$ and  $f_2(X)=f(X+w_2)$. Since $\xi-w_i$ is a simple root of $f_i(X+w_i)$ for $i=1,2$, and for some $i\in \{1,2\}$ we have
${\rm Norm}_{\fqr[\xi-w_i]/\fqr}(\fq[\xi-w_i])={\rm Norm}_{\fqr/\fqr}(\fqr)=\fqr$, that is, Lemma \ref{lem2.1} is applicable and the result follows.
\end{proof}

The following result is immediate from Corollary \ref{cor2.2}.
We call two subspaces $\mathcal{V}$ and $\mathcal{W}$ of $\fqr$ {\em complementary} if
$$\mathcal{V}+\mathcal{W}=\{v+w: v\in \mathcal{V}, w\in \mathcal{W}\}=\fqr.$$
\begin{cor}\label{cor2.4}
Let $r$ be an even positive integer and $\chi$ be a nontrivial multiplicative character of $\fqr$. Let $\mathcal{H}$ be a complementary subspace of $\mathbb{F}_{q^{r/2}}$ in $\fqr$ over $\fq$. Let $\mathcal{W}$ be a subset of $\mathcal{H}$. Let $f(X)$ be a polynomial over $\fqr$ with 
a simple root $\xi\in \fqr$. Let $\B$ be the set of pairs $(w_1,w_2)\in\mathcal{W}^2$ such that $f(X+w_1)$ and $f(X+w_2)$ share conjugated roots over $\mathbb{F}_{q^{r/2}}$. Then, for any $(w_1,w_2)\in\mathcal{W}^2\setminus \B$, we have
\begin{align}\label{c2-1}
\left|\sum_{a\in\mathbb{F}_{q^{r/2}}}\chi(f(a+w_1))\overline{\chi}(f(a+w_2))\right|\le(4D-1)q^{r/4},
\end{align}
where $D$ is the degree of the largest square-free divisor of $f(X)$. 
Moreover, we have
$$\#\B\le 2D^2\#\mathcal{W}.$$
\end{cor}
\begin{proof}
Let $(w_1,w_2)\in\mathcal{W}^2\setminus \B$. By definition $f(X+w_1)$ and $f(X+w_2)$ share no conjugated roots over $\mathbb{F}_{q^{r/2}}$. We claim that $\xi-w_i\not\in\mathbb{F}_{q^{r/2}}$ for some $i\in\{1,2\}$, and thus $\xi-w_i$ is a defining element of $\fqr$ over $\mathbb{F}_{q^{r/2}}$. Suppose $\xi-w_i\in\mathbb{F}_{q^{r/2}}$ for $i=1,2$. Then $w_1-w_2=(\xi-w_2)-(\xi-w_1)\in\mathbb{F}_{q^{r/2}}$. However, we have $w_1-w_2\in\mathcal{H}\setminus\{0\}\subset \fqr\setminus \mathbb{F}_{q^{r/2}}$. 
Therefore, Corollary \ref{cor2.2} implies \eqref{c2-1}.

Note that the polynomials $f(X+w_1)$ and $f(X+w_2)$ have both at most $D$ distinct roots. For any of the $\#\mathcal{W}$ fixed choices of $w_1\in \mathcal{W}$ there are at most $2D^2$ different $w_2$ such that one
of the roots of $f(X+w_2)$ is conjugated to a root of $f(X+w_2)$ and the second result follows.
\end{proof}

\section{$\mathcal{G}$ with restricted coordinates}
In this section, we focus on the first class of subsets, which consist of elements of $\fqr$ with restricted coordinates relative to a given basis of  $\fqr$ over $\fq$. 

Let $\mathcal{C}$ be a collection of $r$ subsets of $\fqr$. We call $\mathcal{C}$ a set of {\it$\fq$-affine hyperplanes in general position}, if for all $1\le k\le r$ the intersection of any $k$ distinct elements of $\mathcal{C}$ is an $\fq$-affine space of dimension $r-k$.

Let $r$ be an even positive integer. Let $\A=\{\mathcal{A}_i\}_{i=1}^r$ be a family of subsets of $\fq$. In this section, we define the subset $\mathcal{G}_{\A}$ as follows
\begin{align}\label{c3-0}
\mathcal{G}_{\A}=\{a_1\alpha_1+\cdots+a_r\alpha_r: a_1\in\mathcal{A}_1,\ldots, a_r\in\mathcal{A}_r \},
\end{align}
where $\alpha_1,\ldots,\alpha_{r/2}$ is a given basis of $\mathbb{F}_{q^{r/2}}$ over $\fq$, and $\alpha_1,\ldots,\alpha_r$ is a completion of $\alpha_1,\ldots,\alpha_{r/2}$ to a basis of $\fqr$ over $\fq$. More generally, we can take any basis of $\fqr$ which contains a basis of $\mathbb{F}_{q^{r/2}}$.
It can be shown that the sets $\G_{\A}$ defined in \eqref{c3-0} with $|\A_1|=\ldots=|\A_r|=q-1$ are exactly the sets avoiding affine hyperplanes in general position, see the introduction of \cite{[GW]}.

Now, we present the first result, which gives an upper bound on the character sum over the subset $\G_{\A}$. 

\begin{thm}\label{thm3.1}
Let $r$ be an even integer, and $\chi$ be a nontrivial multiplicative character of the finite field $\fqr$. Let $\{\alpha_1, \ldots, \alpha_r\}$ be a basis of $\fqr$ over $\fq$, which contains a basis $\{\alpha_1,\ldots,\alpha_{r/2}\}$ of $\mathbb{F}_{q^{r/2}}$ over $\fq$. For a family $\A= \{A_i\}_{i=1}^r$ of subsets of $\fq$, define $\G_{\A}$ as in \eqref{c3-0}. Let $f(X)$ be a polynomial over $\fqr$ with 
a simple root in $\fqr$. Then, we have 
$$|S(\G_{\A},\chi,f)|\le(\#\G_{\A})^{1/2}q^{r/8}\left(\prod_{i=r/2+1}^{r}\#\A_i(4D-1)+2D^2q^{r/4}\right)^{1/2},$$
where $D$ is the degree of the largest square-free divisor of $f(X)$.
\end{thm}
\begin{proof}
First, we write $\G_{\A}$ as a direct sum $\G_{\A}=\mathcal{V}\oplus \mathcal{W}$, where
$$\mathcal{V}=\{a_1\alpha_1+\cdots+a_{r/2}\alpha_{r/2}: a_1\in\mathcal{A}_1,\ldots, a_{r/2}\in\mathcal{A}_{r/2} \}$$
and 
$$\mathcal{W}=\{a_{r/2+1}\alpha_{r/2+1}+\cdots+a_r\alpha_r: a_{r/2+1}\in\mathcal{A}_{r/2+1},\ldots, a_r\in\mathcal{A}_r \}.$$
Then we can express the character sum $S(\G_{\A},\chi,f)$ as follows:
$$S(\G_{\A},\chi,f)=\sum_{\gamma\in\G}\chi(f(\gamma))=\sum_{v\in \mathcal{V}}\sum_{w\in \mathcal{W}}\chi(f(v+w)).$$
By the Cauchy-Schwarz inequality, we obtain 
\begin{align*}
|S(\G_{\A},\chi,f)|^2&\le\# \mathcal{V}\sum_{v\in \mathcal{V}}\left|\sum_{w\in \mathcal{W}}\chi(f(v+w))\right|^2\\
&\le\# \mathcal{V}\sum_{v\in\mathbb{F}_{q^{r/2}}}\left|\sum_{w\in \mathcal{W}}\chi(f(v+w))\right|^2\\
&\le\# \mathcal{V}\sum_{w_1,w_2\in \mathcal{W}}\left|\sum_{v\in\mathbb{F}_{q^{r/2}}}\chi(f(v+w_1))\overline{\chi(f(v+w_2))}\right|.
\end{align*}
Next, we bound the inner sum. Let $\B$ be defined as in Corollary \ref{cor2.4}. For each of the at most $2\# \mathcal{W}D^2$ pairs
$(w_1,w_2)\in\B$ we use the trivial bound $q^{r/2}$. Otherwise, we use Corollary \ref{cor2.4} to bound this sum. Therefore, we have
\begin{eqnarray*}
|S(\G_{\A},\chi,f)|^2&\le& \# \mathcal{V}\left(\#\mathcal{W}^{2}(4D-1)q^{r/4}+2D^2\# \mathcal{W}q^{r/2}\right)\\
&\le &\#\G_{\A} q^{r/4}\left(\prod_{i=r/2+1}^{r}\#\A_i(4D-1)+2D^2q^{r/4}\right),
\end{eqnarray*}
which completes the proof.
\end{proof}

Clearly, for fixed $D$ and sufficiently large $r$, the estimate of Theorem \ref{thm3.1} is nontrivial if $\#\mathcal{A}_i>q^{1/2}$ for any $1\le i\le r$. As a direct application of Theorem \ref{thm3.1}, we obtain the following nontrivial bound for $q=3$.
\begin{cor}\label{cor3.2}
Let $r$ be an even number with 
$$r>4\frac{\log(2D^2)}{\log(4/3)}.$$ Let $\A=\{\A_i\}_{i=1}^r$ be a family of subsets of $\mathbb{F}_3$ of cardinality $\#\A_i=2$, $i=1,\ldots,r$, and $\G_{\A}$ be defined as before. Then, for any polynomial $f(X)\in\fqr[X]$ with a simple root in $\fqr$, we have 
$$|S(\G_{\A},\chi,f)|<2\sqrt{D}\cdot2^{\delta r},$$
where $D$ is the degree of the largest square-free divisor of $f(X)$, and
$$\delta=\frac{3}{4}+\frac{\log(3)}{8\log(2)}=0.94812\ldots.$$
\end{cor}
\begin{proof}
By Theorem \ref{thm3.1}, we obtain that 
\begin{align*}|S(\G_{\A},\chi,f)|\le 
2^{r/2}\cdot 3^{r/8}\left(4D4^{r/4}-4^{r/4}+2D^2\cdot 3^{r/4}\right)^{1/2}.
\end{align*}
Note that $4^{r/4}-2D^2\cdot 3^{r/4}>0$ by the condition on $r$.
It follows that 
$$|S(\G_{\A},\chi,f)|<2^{r/2}\cdot 3^{r/8}\left(4D4^{r/4}
\right)^{1/2}=2\sqrt{D}\cdot2^{\delta r},$$
as desired.
\end{proof}
\begin{rem}\label{rem2.4}
Let $q=3$, and let $\mathcal{A}_i=\{0,2\}$ for any $1\le i\le r$. Let 
$$\G_{\A}=\left\{\sum_{i=1}^ra_i\alpha_i: a_i\in\A_i\right\},$$ 
$\chi$ be a multiplicative character of order $d\ge 2$ and $f(X)\in\fqr[X]$ satisfy that $f(X)$ is not a $d$th power of a polynomial over $\fqr$.  Iyer and Shparlinski \cite[Corollary 3.3]{[IS]} derived the following nontrivial bound for $|S(\G,\chi,f)|$:
$$|S(\G_{\A},\chi,f)|\le c\cdot 2^{\gamma r},$$
where $c>0$ is a constant and $\gamma=0.99128\ldots$
Compared to Corollary \ref{cor3.2}, our result offers a sharper upper bound for the character sum $S(\G_{\A},\chi,f)$ under stronger conditions.
\end{rem}

Recently, Grzywaczyk and the second author (see \cite[Theorem 2.1]{[GW]}) provided a new estimate, presented below, which is particularly interesting for small values of $q$.
\begin{lem}\label{lem3.4}
Let $c_1,c_2\ldots,c_r$ be elements of $\fq$, and let $\A=\{\A_i\}_{i=1}^r$ with $\mathcal{A}_i=\fq\setminus\{c_i\}$ for any $1\le i\le r$. Then for 
$$\G_{\A}=\left\{\sum_{i=1}^ra_i\alpha_i: a_i\in\mathcal{A}_i\ \text{for\ each}\ i\right\}$$ and a nontrivial multiplicative character $\chi$ of $\fqr$, one has
$$|S(\G_{\A},\chi)|\le \sqrt{3}(q-1)^{r/2}q^{\lceil 3r/4\rceil/2}.$$
\end{lem}
Let $\G_{\A}$ be the set given in Lemma \ref{lem3.4}. Using Theorem \ref{thm3.1} and taking $f(X)=X$, we obtain a sharper upper bound for the character sum $S(\G_{\A},\chi)$:
\begin{align}\label{c3-2}
|S(\G_{\A},\chi)|\le (q-1)^{3r/4}q^{r/8}\left(3+\frac{2q^{r/4}}{(q-1)^{r/2}}\right)^{1/2}
<2(q-1)^{3r/4}q^{r/8},\quad q\ge 3,~r\ge 10,
\end{align}
since we can easily verify that $2q^{r/4}<(q-1)^{r/2}$ for $r\ge 10$ and $q\ge 3$. This bound improves upon the result of Lemma \ref{lem3.4} for $q\ge 3$ and holds asymptotically as $r$ increases.

Grzywaczyk and the second author applied Lemma \ref{lem3.4} to prove that the set $\G_{\A}$ contains a primitive element of $\fqr$ when $q=4$ or $q=5$, provided $r$ is large enough. However, their result is not applicable for the case $q=3$. Our improved upper bound of (\ref{c3-2}) can be used to show that $\G_{\A}$ contains a primitive element of $\fqr$ for $q\ge 3$, assuming $r$ is a sufficiently large even integer. 
\begin{thm}\label{thm3.5}
Let $r$ be even and $\G_{\A}$ be defined as in Lemma \ref{lem3.4} with a basis $\{\alpha_1,\ldots,\alpha_r\}$ which contains a basis of $\mathbb{F}_{q^{r/2}}$, and let $N(\G_{\A})$ denote the number of primitive elements in $\G_{\A}$. Then we have the following lower bound: 
$$N(\G_{\A})>\frac{\phi(q^r-1)(q-1)^{3r/4}}{q^r-1}\left((q-1)^{1/4}-2q^{1/8}W(q^r-1)\right),$$
where $\phi$ is the Euler totient function, and $W(q^r-1)$ is the number of square-free divisors of $q^r-1$.
\end{thm}
\begin{proof}
By applying the well-known Vinogradov formula (see \cite[Exercise 5.14]{[LN]}), we derive 
$$
\frac{q^r-1}{\phi(q^r-1)}N(\G_{\A})=\sum_{e\mid (q^r-1)}\frac{\mu(e)}{\phi(e)}\sum_{\chi\in\Lambda_e}\sum_{w\in\G_{\A}}\chi(w),$$
where $\mu$ is the M\"obius function and $\Lambda_e$ is the set of all multiplicative characters of $\fqr$ of order $e$. Using the upper bound of (\ref{c3-2}), we get
\begin{align*}
\frac{q^r-1}{\phi(q^r-1)}N(\G_{\A})
&>(q-1)^r-1-2(q-1)^{3r/4}q^{r/8}\sum_{\substack{2\le e\mid (q^r-1)}}1\\
&>(q-1)^r-2(q-1)^{3r/4}q^{r/8}W(q^r-1),
\end{align*}
which yields the desired inequality. 
\end{proof}
\begin{lem}\label{lem3.6}\cite{[CST]}
Let $W(t)$ represent the number of square-free divisors of $t$. Then for $t\ge 3$ one has that
$$W(t-1)<t^{\frac{0.96}{\log\log t}}.$$
\end{lem}
Thus, from Theorem \ref{thm3.5} and Lemma \ref{lem3.6}, we obtain the following result.
\begin{cor}\label{cor3.7}
Let $c_1,\ldots,c_r$ be elements of $\fq$, and let $\A=\{\A_i\}_{i=1}^r$ with $\mathcal{A}_i=\fq\setminus\{c_i\}$ for each $i$. Then for any $q\ge 3$ the set $\G_{\A}$ contains a primitive element of $\fqr$ if $r$ is a large enough even integer. Precisely, for some small $q$ and even $r$, $\G_{\A}$ contains a primitive element of $\fqr$ provided one of the following holds:
\begin{enumerate}[label=(\alph*)]
    \item $q=9$ and $r\ge 2504$;
    \item $q=8$ and $r\ge 3256$;
    \item $q=7$ and $r\ge 4754$;
    \item $q=5$ and $r\ge 25662$;
    \item $q=4$ and $r\ge 363184$;
    \item $q=3$ and $r\ge 7951\cdot 10^8$.
\end{enumerate}
\end{cor}
\begin{proof}
Theorem \ref{thm3.5} and Lemma \ref{lem3.6} provide a sufficient condition for the existence of a primitive element in \( \G_{\A} \):
\begin{align}\label{c3-3}
\left( 2q^{\frac{0.96r}{\log\log q^r}} \right)^{\frac{1}{r}} \leq \left( \frac{(q-1)^2}{q} \right)^{\frac{1}{8}}.
\end{align}
Note that the term \( \left( 2q^{\frac{0.96r}{\log\log q^r}} \right)^{\frac{1}{r}} \) is monotonically decreasing and tends to 1 as \( r \to \infty \), while \( \left( \frac{(q-1)^2}{q} \right)^{\frac{1}{8}} > 1 \) for all \( q \geq 3 \). Thus, the inequality holds for sufficiently large \( r \). Moreover, for \( q = 4, 5, 7, 8, \) or \( 9 \), a computer-assisted check verifies that inequality \eqref{c3-3} holds when \( q = 9 \) and \( r \geq 2504 \), \( q = 8 \) and \( r \geq 3256 \), \( q = 7 \) and \( r \geq 4754 \), \( q = 5 \) and \( r \geq 25662 \), \( q = 4 \) and \( r \geq 363184 \), and $q=3$ and $r\ge 7951\cdot 10^8$. Hence, the result follows.
 \end{proof}
 The following table (see TABLE \ref{table1}) displays the performance of three estimates in establishing the existence of primitive elements in the finite field $\fqr$ avoiding affine hyperplanes. We used the term 'large enough' when the smallest suitable $r$ was at least $10^{12}$ and not exactly computable without too much effort.
 \begin{table}[h]
\centering
\caption{Values of $(q,r)$ for which the three estimates are applicable}
\label{table1}
\begin{tabular}{|c|c|c|c|}
\hline
\( q \) & \makecell{ New estimate\\ (even $r$)}  & \makecell{Estimate of \cite{[GW]}\\ (any $r$)} & \makecell{Estimate of \cite{[FR]}\\ (any $r$)}  \\
\hline
9 & $r\ge 2504$  & $r\ge 8040$ & $r$ large enough   \\
8& $r\ge 3256$ & $r\ge 14831$ & $r$ large enough\\
7 & $r\ge 4754$ & $r\ge 39247$ & $r$ large enough \\
5& $r\ge 25662$ & $r\ge 1.92\cdot 10^7$ & ---  \\
4& $r\ge 363184$ & $r$ large enough & ---\\
3& $r\ge7.96\cdot 10^{11}$ & $r$ large enough & ---\\
\hline
\end{tabular}
\end{table}
Therefore, the lower bound we obtained for $N(\G_{\A})$ is significantly better than the results of Fernandes and Reis \cite{[FR]} and Grzywaczyk and the second author \cite{[GW]} when $q$ is small. 

\section{$\G$ with $s$-sparse elements}
Let $r$ be an even number. Let $\{\alpha_1,\ldots,\alpha_{r/2}\}$ be a basis of $\mathbb{F}_{q^{r/2}}$ over $\fq$, and $\{\alpha_1,\ldots,\alpha_r\}$ be a completion of $\{\alpha_1,\ldots,\alpha_{r/2}\}$ to a basis of $\fqr$ over $\fq$. We define the \textit{weight} $wt(\xi)$ of $\xi\in\fqr$ by 
$$wt(\xi)=\#\{1\le i\le r: a_i\ne 0\}\quad \mbox{if}\quad \xi=a_1\alpha_1+\cdots+a_r\alpha_r$$ with each $a_i\in\fq$. For a given positive integer $s<r$ put
\begin{align}\label{c4-0}
\mathcal{G}_s=\{\xi\in\fqr: wt(\xi)=s\}.
\end{align}
We call $\G_s$ the set of {\em $s$-sparse elements} of $\fqr$ with respect to the basis $\{\alpha_1,\ldots,\alpha_r\}$. In this section, we will estimate the sum $S(\G_s,\chi,f)$, where $\chi$ is a nontrivial multiplicative character of $\fqr$ and $f(X)$ is a polynomial over $\fqr$. We focus on the most interesting case $q=2$. 

Let us begin by introducing preliminary concepts. Let $H(x)$ be the \textit{binary entropy} function defined by 
\begin{equation}\label{entropy}
H(x)=-x\log_2x-(1-x)\log_2(1-x),\ 0<x<1,
\end{equation}
and $H(0)=H(1)=0$. We also define 
$$H^*(x)=\begin{cases}
H(x),&\text{if}\ 0\le x\le \frac{1}{2},\\
1,&\text{if}\ x>\frac{1}{2}.
\end{cases}$$
The following lemma on the size of binomial coefficients will be used to bound the cardinality of some sets.
\begin{lem}\label{lem4.1} \cite[Chapter 10, Corollary 9]{[MS]}
For any positive integer $n$ and a real~$\gamma$ with $0<\gamma\le \frac{1}{2}$, we have
$$\sum_{0\le m\le\gamma n}\binom{n}{m}\le 2^{nH(\gamma)}.$$
Consequently, 
$$\binom{n}{m}\le2^{nH^*(m/n)}$$
for any positive integers $m$ and $n$ with $m\le n$.
\end{lem}

As usual we use the notation 
$$A(r) \ll B(r)\quad \mbox{if}\quad |A(r)| \leq c \cdot B(r)\quad \mbox{for some constant}\ c > 0$$
and
$$A(r) = o(B(r))\quad \mbox{if} \quad \lim_{r \to \infty} \frac{A(r)}{B(r)} = 0.$$
The implied constants depend only on the degree of a fixed polynomial.

Let $\mathrm{f, g, h}$ be three functions in two variables $\rho$ and $\lambda$ with $0<\lambda\le\rho\le \frac{1}{2}$
defined by
\begin{align}\label{c4-1-1}
&\mathrm{f}(\rho,\lambda)=\frac{1}{2}H^{*}(2\lambda)+\frac{1}{2}H^*(2\rho-2\lambda),\\\label{c4-1-2}
&\mathrm{g}(\rho,\lambda)=\frac{1}{8}+\frac{1}{2}H(\rho)+\frac{1}{4}H^*(2\rho-2\lambda),\\\label{c4-1-3}
&\mathrm{h}(\rho,\lambda)=\frac{1}{2}H(\rho)+\frac{1}{4}.
\end{align}

Now we can state the main result of this section
as follows.
\begin{thm}\label{thm4.2}
Let $r$ be an even number and $q=2$. Let $\chi$ be a nontrivial multiplicative character of the finite field $\fqr$. Let $\{\alpha_1, \ldots, \alpha_r\}$ be a basis of $\fqr$ over $\fq$, such that $\{\alpha_1, \ldots, \alpha_{r/2}\}$ forms a basis of $\mathbb{F}_{q^{r/2}}$ over $\fq$. For any $s$ with $1\le s<r$, 
let $\G_s$ be defined as in \eqref{c4-0}, and put $\rho=\frac{\min\{s,r-s\}}{r}$. Then, for any polynomial $f(X)\in\fqr[X]$ with a simple root in $\fqr$, we have
$$|S(\G_s,\chi,f)|\le 2^{\eta'(\rho)r+o(r)},$$
where
$$\eta'(\rho)=\min_{0<\lambda\le\frac{1}{2}\rho}\max\{\mathrm{f}(\rho,\lambda),\mathrm{g}(\rho,\lambda),\mathrm{h}(\rho,\lambda)\},$$
where $\mathrm{f}$, $\mathrm{g}$ and $\mathrm{h}$ are defined by \eqref{c4-1-1}, \eqref{c4-1-2} and \eqref{c4-1-3}, respectively,
and the implied constant depends only on the degree of $f(X)$.
\end{thm}
\begin{proof}
We may restrict ourselves to the case $s\le \frac{r}{2}$, that is $\rho\le \frac{1}{2}$, since otherwise we may consider the polynomial $f(X+\alpha_1+\ldots+\alpha_r)$ and the set $\G_{r-s}$ instead of $f(X)$ and~$\G_s$. Both lead to the same character sums since $wt(\xi+\alpha_1+\ldots+\alpha_r)=r-wt(\xi)$.

Now put
$$U_1=\left\{a_1\alpha_1+\cdots+a_{r/2}\alpha_{r/2}: a_i\in\ftwo\ \text{for\ any $i$ with}\ 1\le i\le \frac{r}{2}\right\}$$
and 
$$U_2=\left\{a_{r/2+1}\alpha_{{r/2}+1}+\cdots+a_r\alpha_r: a_i\in\ftwo\ \text{for\ any $i$ with}\ {\frac{r}{2}}+1\le i\le r\right\}.$$
We decompose $\G_s$ into the disjoint union of $s+1$ subsets as 
$$\G_s=\bigcup_{i=0}^{s}\left\{u_1+u_2: u_1\in U_1^{(i)}\ \text{and}\ u_2\in U_2^{(s-i)}\right\},$$
where $U_j^{(i)}=\{u\in U_j: wt(u)=i\}$ for any $j=1,2$ and $0\le i\le s$. It follows that
$$S(\G_s,\chi,f)=\sum_{i=0}^{s}\M_i,$$
where 
$$\M_i=\sum_{u_1\in U_1^{(i)}}\sum_{u_2\in U_2^{(s-i)}}\chi(f(u_1+u_2)).$$
Put $t=\lfloor\lambda r \rfloor$, where $\lambda$ with $0<\lambda\le\rho\le\frac{1}{2}$ is a variable. In the following, we estimate~$|\M_i|$. 

For $0\le i<t$ we use  the trivial bound
$$|\M_i|\le \#U_{1}^{(i)}\#U_{2}^{(s-i)}=\binom{\frac{r}{2}}{i}\binom{\frac{r}{2}}{s-i}.$$
Then by Lemma \ref{lem4.1} we obtain 
$$|\M_i|\le 2^{\tilde{\mathrm{f}}(\tau)r},$$ 
where $\tau=\frac{i}{r}$ and
$$\tilde{\mathrm{f}}(\tau)=\frac{1}{2}(H^*(2\tau)+H^*(2\rho-2\tau))\quad 0\le \tau\le \lambda.$$ 
It is clear that $\tilde{\mathrm{f}}(\tau)$ attains its maximum at $\tau = \min\{\lambda, \frac{\rho}{2}\}$ when $\tau$ varies over $0 < \tau \leq \lambda$ (see, for instance, \cite[Lemma 3.6]{[MSW]}). So we may assume $\lambda\le \frac{\rho}{2}$ and write 
$$\mathrm{f}(\rho,\lambda)=\tilde{\mathrm{f}}(\lambda)=\frac{1}{2}(H^*(2\lambda)+H^*(2\rho-2\lambda)).$$ 
Then we get 
\begin{align}\label{c4-4}
\sum_{i=0}^{t-1}|\M_i|\le t\cdot 2^{\mathrm{f}(\rho,\lambda)r}=2^{\mathrm{f}(\rho,\lambda)r+o(r)}.
\end{align}

Next, we proceed to estimate $|\M_i|$ for $i\ge t$. Note that
$$|\M_i|\le\sum_{u_1\in U_1^{(i)}} \left|\sum_{u_2\in U_2^{(s-i)}}\chi(f(u_1+u_2))\right|.
$$
By the Cauchy-Schwarz inequality 
we get
\begin{align}\label{c4-2}
|\M_i|^2\le\#U_1^{(i)}\sum_{u_1\in U_1^{(i)}}\left|\sum_{u_2\in U_2^{(s-i)}}\chi(f(u_1+u_2))\right|^2.
\end{align}
For the sums on the right-hand side of (\ref{c4-2}), we extend the first sum over $U_1^{(i)}$ to a sum over $U_1=\mathbb{F}_{q^{r/2}}$, and derive
\begin{align*}
\sum_{u_1\in U_1^{(i)}}\left|\sum_{u_2\in U_2^{(s-i)}}\chi(f(u_1+u_2))\right|^2&\le\sum_{u_1\in\mathbb{F}_{q^{r/2}}}\left|\sum_{u_2\in U_2^{(s-i)}}\chi(f(u_1+u_2))\right|^2\\&=\sum_{u^{\prime}_2,u^{\prime\prime}_2\in U_2^{(s-i)}}\left|\sum_{u_1\in\mathbb{F}_{q^{r/2}}}\chi(f(u_1+u^{\prime}_2))\overline{\chi}(f(u_1+u^{\prime\prime}_2))\right|.
\end{align*}
We bound the inner sum by Corollary \ref{cor2.4} for the pairs $(u^{\prime}_2,u^{\prime\prime}_2)\in U_2^{(s-i)}\times U_2^{(s-i)}$ for which $f(X+u^{\prime}_2)$ and $f(X+u^{\prime\prime}_2)$ share no conjugate root over $\mathbb{F}_{q^{r/2}}$. For the other pairs $(u^{\prime}_2,u^{\prime\prime}_2)$ we use the trivial bound $q^{r/2}$. Thus,
\begin{align*}
&\sum_{u_1\in U_1^{(i)}}\left|\sum_{u_2\in U_2^{(s-i)}}\chi(f(u_1+u_2))\right|^2\ll (\#U_2^{(s-i)})^22^{r/4}+\#U_2^{(s-i)}2^{r/2}.
\end{align*}
It then follows from (\ref{c4-2}) that
$$|\M_i|^2\ll\#U_1^{(i)}\#U_2^{(s-i)}\left(\#U_2^{(s-i)}2^{r/4}+2^{r/2}\right).$$
Together with $\#U_1^{(i)}\#U_2^{(s-i)}\le\#\G_s$ we get
$$|\M_i|^2\ll\#\G_s\#U_2^{(s-i)}2^{r/4}+\#\G_s2^{r/2}.$$
Finally, by Lemma \ref{lem4.1} we have
$$\# \G_s=\binom{r}{s}\le 2^{H(\rho)r} \quad \mbox{and}\quad \# U_2^{(s-i)}=\binom{\frac{r}{2}}{s-i}\le 2^{H^*(2(\rho-\lambda)) r/2+o(r)},$$
and obtain
$$|\M_i|\ll(2^{\mathrm{g}(\rho,\lambda)r}+2^{\mathrm{h}(\rho,\lambda)r})2^{o(r)},$$
where $ \mathrm{g}(\rho,\lambda)$ and $\mathrm{h}(\rho,\lambda)$ are defined as in (\ref{c4-1-2}) and (\ref{c4-1-3}), and then we have
\begin{align}\label{c4-7}
\sum_{i=t}^{s}|\M_i|\ll (s-t+1)(2^{\mathrm{g}(\rho,\lambda)r}+2^{\mathrm{h}(\rho,\lambda)r})2^{o(r)}\ll(2^{\mathrm{g}(\rho,\lambda)r}+2^{\mathrm{h}(\rho,\lambda)r})2^{o(r)}.
\end{align}
Therefore, combining (\ref{c4-4}) and (\ref{c4-7}) together, we arrive at
the final result. 
\end{proof}

\section{Conclusion}
In this paper, we continued the study of multiplicative character sums over subsets of finite fields by focusing on two distinct classes of subsets in quadratic extensions $\fqr$, where $r$ is an even integer. Our primary contribution is a new and sharper estimate for the character sums $S(\mathcal{G}, \chi, f)$, reducing them to character sums over subfields, specifically $\mathbb{F}_{q^{r/2}}$, rather than over arbitrary linear spaces as in previous works. This approach has allowed us to improve earlier bounds of \cite{[GW],[IS],[MSW]}, subject to the condition that the basis $\{\alpha_1, \ldots, \alpha_r\}$ of $\fqr$ contains a basis 
of its subfield $\mathbb{F}_{q^{r/2}}$.

\section*{Acknowledgment}
This research was partially supported by the China Scholarship Council Fund (Grant No.\ 202301010002), the Scientific Research Innovation Team Project of China West Normal University (Grant No.\ KCXTD2024-7), and the Austrian Science Fund (FWF) (Grant No.\ 10.55776/PAT4719224).

The authors would like to thank Daqing Wan and Chi Hoi Yip for pointing to the correction of \cite{[Wan]} presented in \cite{[MY]}.

\end{document}